\documentclass[11pt]{amsart}
\usepackage{amsmath,amssymb,euscript,mathrsfs,pstricks}
\usepackage[small,heads=LaTeX,nohug]{diagrams}
\usepackage{hyperref}

\psset{unit=1pt}
\psset{arrowsize=4pt 1}
\psset{linewidth=.5pt}

\newcommand{\define}{\textbf}

\newcommand{\isom}{\cong}
\renewcommand{\setminus}{\smallsetminus}
\renewcommand{\phi}{\varphi}
\newcommand{\exterior}{\textstyle\bigwedge}
\renewcommand{\tilde}{\widetilde}
\renewcommand{\hat}{\widehat}
\renewcommand{\bar}{\overline}

\newcommand{\C}{\mathbb{C}}

\newcommand{\Z}{\mathbb{Z}}
\renewcommand{\P}{\mathbb{P}}

\renewcommand{\O}{\mathcal{O}}
\newcommand{\lieg}{\mathfrak{g}}
\newcommand{\liep}{\mathfrak{p}}
\newcommand{\lieb}{\mathfrak{b}}
\newcommand{\liet}{\mathfrak{t}}

\newcommand{\<}{\langle}
\renewcommand{\>}{\rangle}

\newcommand{\Fl}{{Fl}}
\newcommand{\Gr}{{Gr}}
\newcommand{\FFl}{\mathbf{Fl}}
\newcommand{\GGr}{\mathbf{Gr}}

\newcommand{\OOmega}{\mathbf{\Omega}}

\DeclareMathOperator{\Sym}{Sym}

\DeclareMathOperator{\rk}{rk}
\DeclareMathOperator{\diag}{diag}

\DeclareMathOperator{\End}{End}
\DeclareMathOperator{\Hom}{Hom}

\newtheorem{theorem}{Theorem}[section]
\newtheorem{lemma}[theorem]{Lemma}
\newtheorem{proposition}[theorem]{Proposition}
\newtheorem{corollary}[theorem]{Corollary}

\theoremstyle{definition}
\newtheorem{definition}[theorem]{Definition}
\newtheorem{remark}[theorem]{Remark}

\begin{document}

\title{Degeneracy of triality-symmetric morphisms}
\author{Dave Anderson}
\address{Department of Mathematics\\University of Michigan\\Ann Arbor, MI 48109}
\email{dandersn@umich.edu}
\date{December 3, 2008}
\thanks{This work was partially supported by an RTG fellowship, NSF Grant 0502170.}

\begin{abstract}
We define a new symmetry for morphisms of vector bundles, called \emph{triality symmetry}, and compute Chern class formulas for the degeneracy loci of such morphisms.  In an appendix, we show how to canonically associate an octonion algebra bundle to any rank $2$ vector bundle.
\end{abstract}

\maketitle

\section{Introduction}\label{sec:intro}

Let $\phi:E \to F$ be a morphism of vector bundles on a smooth variety $X$, of respective ranks $m$ and $n$.  The $r$th \emph{degeneracy locus} of $\phi$ is the set of points of $X$ defined by
\begin{eqnarray*}
D_r(\phi) = \{x\in X\,|\, \rk \phi(x) \leq r \},
\end{eqnarray*}
where $\phi(x):E(x)\to F(x)$ is the corresponding linear map in the fibers over $x\in X$.  Such loci are ubiquitous in algebraic geometry: many interesting varieties, from Veronese embeddings of projective spaces to Brill--Noether loci parametrizing special divisors in Jacobians, can be realized as degeneracy loci for appropriate maps of vector bundles.  General geometric information about degeneracy loci is therefore often useful.  In particular, one can ask for Chern class formulas for the cohomology class of $D_r(\phi)$ in $H^*X$ --- what is $[D_r(\phi)]$ as a polynomial in the Chern classes of $E$ and $F$?

When $\phi$ is sufficiently general, so $D_r(\phi)$ has \emph{expected codimension} equal to $(m-r)(n-r)$, the answer is given by the Giambelli--Thom--Porteous determinantal formula.  In two cases of particular interest, Chern class formulas are known for degeneracy loci where $\phi$ is not general in this sense.  Taking $F=E^*$, one has the dual morphism $\phi^*:E^{**} = E \to E^*$.  Call $\phi$ \emph{symmetric} if $\phi^*=\phi$, and \emph{skew-symmetric} if $\phi^*=-\phi$.  The codimension of $D_r(\phi)$ is at most $\binom{m-r+1}{2}$ (in the symmetric case) or $\binom{m-r}{2}$ (in the skew-symmetric case), so such morphisms are never sufficiently general for the Giambelli--Thom--Porteous formula to apply.  Formulas for these loci were given by Harris--Tu \cite{ht} and J\'ozefiak--Lascoux--Pragacz \cite{jlp}.  As explained in \cite{fnr}, these formulas can also be found by computing the equivariant classes of appropriate orbit closures in the $GL(E)$-representations $\Sym^2 E^*$ and $\exterior^2 E^*$, where $E$ is a vector space.  See \cite[Chapter~6]{fp} for more detailed discussions of the formulas.

The primary goal of the present article is to give degeneracy locus formulas for a new class of morphisms, which we call \emph{triality-symmetric} morphisms.  Letting $E$ be a rank $2$ vector bundle, these are maps
\begin{eqnarray*}
\phi: E \to End(E) \oplus E^*
\end{eqnarray*}
possessing a certain symmetry related to the $S_3$ symmetry of the $D_4$ Dynkin diagram.  Specifically, we use the following definition:
\begin{definition} \label{def:triality-symmetry}
A morphism $\phi:E \to End(E)\oplus E^*$ is \define{triality-symmetric} if the corresponding section of $Hom(E,End(E)\oplus E^*)$ lies in
\begin{eqnarray*}
(Sym^3 E^* \otimes \exterior^2 E) \oplus \exterior^2 E^*.
\end{eqnarray*}
That is, $\phi=\phi_1\oplus\phi_2$, with $\phi_1$ defining a symmetric trilinear form $Sym^3 E\to\exterior^2 E$ and $\phi_2$ defining an alternating bilinear form $\exterior^2 E \to \O_X$.
\end{definition}

A few words of motivation are in order concerning this definition.  For simplicity, consider the case where $X$ is a point.  The space of all linear maps $\Hom(E,F)$ is also the tangent space to the Grassmannian $Gr(m,m+n) = Gr(m,E\oplus F) = GL_{m+n}/P$ (for an appropriate maximal parabolic subgroup $P$) at the point corresponding to $E$.  When $F=E^*$, there is a canonical symplectic form $\omega$ on $E\oplus E^*$, defining the Lagrangian Grassmannian $LG(m,2m) \subseteq Gr(m,2m)$, and the space of symmetric morphisms $\Sym^2 E^*$ is naturally identified with the tangent space to $LG(m,2m) = Sp_{2m}/P$ at the point $[E]$.  Moreover, $LG(m,2m)$ is the fixed locus for the involution of $Gr(m,2m)$ which sends a subspace to its orthogonal complement under $\omega$.  The situation is similar for skew-symmetric morphisms, replacing the Lagrangian Grassmannian with the orthogonal Grassmannian $OG(m,2m) = SO_{2m}/P$.

From this point of view, it is natural to expect nice degeneracy locus formulas corresponding to other finite symmetries of homogeneous spaces.  A particularly interesting one is the \emph{triality} action on $OG(2,8)$.  A concise description of this $S_3$ action may be found in \cite[Appendix B]{thesis}; for more details, see \cite{vdbs} or \cite{gari-triality}.  For our purposes, the relevant facts are that the fixed locus is the ``$G_2$ Grassmannian'' $G_2/P$ (for $P$ corresponding to the long root), and the tangent space to $G_2/P$ is naturally identified with $(\Sym^3 E^* \otimes \exterior^2 E) \oplus \exterior^2 E^*$ at the point $[E] \in G_2/P \subseteq OG(2,8)$.  (In \S\ref{sec:triality}, we will explicitly exhibit the $S_3$ action on the tangent space $T_{[E]}OG(2,8) \isom \Hom(E,\End(E)) \oplus \exterior^2 E^*$ fixing $(\Sym^3 E^* \otimes \exterior^2 E) \oplus \exterior^2 E^*$.)  Further motivation comes from the fact that there is a canonical \emph{octonion algebra} structure on $E\oplus End(E) \oplus E^*$, when $E$ is a rank $2$ vector bundle, just as there is a canonical symplectic structure on $E\oplus E^*$.  This is the content of Proposition \ref{prop:bundle-constr}.

Since $E$ is required to have rank $2$, a triality-symmetric morphism may have rank $0$, $1$, or $2$.  Write $D_r(\phi) \subseteq X$ for the locus of points where $\phi$ has rank at most $r$.  For a triality-symmetric morphism $\phi$, define the \define{expected codimension} of $D_r(\phi)$ to be $5$, $3$, or $0$ if $r=0$, $r=1$, or $r=2$, respectively.  With this understood, we may state our main theorem:

\begin{theorem} \label{thm:triality-formula}
Let $c_1,c_2$ be the Chern classes of $E^*$, and let $x_1,x_2$ be Chern roots.  Let $\phi:E\to End(E) \oplus E^*$ be a triality-symmetric morphism.  If $D_r(\phi)$ has expected codimension and $X$ is Cohen--Macaulay, then we have $[D_r(\phi)] = P_r(c_1,c_2)$ in $H^*X$, where
\begin{eqnarray*}
P_2 &=& 1 ,\\
P_1 &=&  3\, c_2\, c_1 = 3 x_1 x_2 (x_1+x_2),\\
P_0 &=& c_2\, c_1 \, (9\,c_2 - 2\,c_1^2) = x_1 x_2 (x_1+x_2) (2x_1 - x_2) (-x_1 + 2x_2).
\end{eqnarray*}
\end{theorem}

We will give two proofs, both involving the simple Lie group of type $G_2$, but using substantially different approaches.  The first relates degeneracy loci for triality-symmetric morphisms to certain Schubert loci in a $G_2$ flag bundle, just as Fulton's generalization of the Harris--Tu formulas relates symmetric morphisms to type $C$ flag bundles \cite{orthosymp}.  One then applies the formulas for $G_2$ Schubert loci developed in \cite{g2chern} to derive the formulas of Theorem \ref{thm:triality-formula}.

The second proof uses equivariant cohomology, in the spirit of \cite{fr} and \cite{fnr} (but see Remark \ref{rmk:restriction}).  More precisely, for $P$ the maximal parabolic subgroup of $G_2$ corresponding to the long root and $E$ a two-dimensional vector space, we consider $(\Sym^3 E^* \otimes \exterior^2 E) \oplus \exterior^2 E^*$ as a $P$-module and compute the equivariant classes of the $P$-orbit closures in this vector space.  Certain of these orbit closures correspond to degeneracy loci, and one can deduce Theorem \ref{thm:triality-formula} from the equivariant formulas.  Along the way, we explicitly identify the $P$-orbit closures in $(\Sym^3 E^* \otimes \exterior^2 E) \oplus \exterior^2 E^*$, and compute all their equivariant classes (Proposition \ref{prop:orbit-desc} and Theorem \ref{thm:orbit-formula}).

Triality symmetry is the $G_2$ case of a general notion of symmetry for morphisms of vector bundles; in fact, two types of symmetry for morphisms can be naturally associated to any maximal parabolic subgroup $P$ of a complex reductive group $G$, as described in \cite[Appendix C]{thesis}.  The ``orbit'' approach used in the second proof of Theorem \ref{thm:triality-formula} generalizes to the following problem:
{\it
Compute the equivariant classes of $P$-orbit (or $B$-orbit) closures for the adjoint action on $\lieg/\liep$.}
Solutions to this problem account for many of the known degeneracy locus formulas; see, e.g., \cite{fr-ss} and \cite{knutson-miller}.

A related problem is to classify situations where there are finitely many orbits.  In the case of $P$ acting on $\lieg/\liep$, this problem was investigated by Popov and R\"oehrle \cite{pr}, and such parabolic actions have been classified \cite{buhe,hr,jr}.  The classification of Borel or Levi subgroup actions on $\lieg/\liep$ with finitely many orbits appears to be unknown.

\subsection*{Acknowledgments}

This work is part of my Ph.\ D.\ thesis, and it is a pleasure to thank William Fulton for his encouragement in this project.  Thanks also to Danny Gillam for useful conversations about triality.

\section{Preliminaries}\label{sec:prelim}

All varieties are over $\C$.  We will write $X$ for the base variety.  If $E$ is a vector bundle on $X$, we write $E(x)$ for the fiber over $x\in X$.  We often suppress notation for pullback of vector bundles.

\subsection{Octonions}\label{subsec:octonions}

An \emph{octonion algebra} is an $8$-dimensional vector space $C$, together with a nondegenerate quadratic norm $N$ and a multiplication with unit, such that $N(uv) = N(u) N(v)$ for all $u,v\in C$.  Write $\<\;,\;\>$ for the symmetric bilinear form corresponding to $N$.  The notion of an octonion algebra globalizes easily to \emph{octonion bundles}, where $C$ is a rank $8$ vector bundle on a variety $X$, the multiplication is a vector bundle map $C\otimes C \to C$, and for simplicity we assume the norm takes values in $\O_X$.  For more on octonions and octonion bundles, see \cite[\S\S1--2]{sv}, \cite{petersson}, or \cite[\S2]{thesis}.

The group of algebra automorphisms of an octonion algebra is the simple complex Lie group of type $G_2$; abusing notation, we will write $G_2$ to denote this group.

Let $E$ be a rank two vector bundle on $X$.  By Proposition \ref{prop:bundle-constr}, $C = E \oplus End(E) \oplus E^*$ has a canonical structure of an octonion bundle.  It will be convenient to use a basis adapted to this construction, in the case where $X$ is a point, so $E$ is a $2$-dimensional vector space.  Let $v_1,v_2$ be a basis for $E$, and extend to a basis for $C=E\oplus \End(E) \oplus E^*$ by setting
\begin{eqnarray} \label{v-basis}
\begin{array}{rcl}
v_3 &=& v_2^*\otimes v_1 \\
v_4 &=& v_1^*\otimes v_1 \\
v_5 &=& v_2^*\otimes v_2 \\
v_6 &=& v_1^*\otimes v_2 \\
v_7 &=& v_2^* \\
v_8 &=& v_1^*.
\end{array}
\end{eqnarray}
Thus the identity element is $e=v_4+v_5$. 

With respect to this basis, the symmetric bilinear form $\<\;,\;\>$ is given by
\begin{eqnarray} \label{eqn:v-beta}
\begin{array}{rcl}
\< v_p,v_{9-q} \> &=& -\delta_{pq}, \text{ for }p,q\neq 4,5; \\
\< v_4,v_5 \> &=& 1.
\end{array}
\end{eqnarray}
Write $V = e^\perp \subset C$ for the orthogonal complement of the identity element with respect to $\<\;,\;\>$.  Thus $V$ is defined by $v_4^*+v_5^* = 0$.

Let the torus $T=(\C^*)^2$ acts on $C$ in this basis via the matrix
\begin{eqnarray}\label{t-action}
\diag(z_1, z_2, z_1 z_2^{-1}, 1, 1, z_1^{-1} z_2, z_2^{-1}, z_1^{-1}),
\end{eqnarray}
with weights
\begin{eqnarray}\label{t-weights}
\{t_1, t_2, t_1-t_2, 0, 0, -t_1+t_2, -t_2, -t_1\}.
\end{eqnarray}
This is induced from the standard action on $E = \mathrm{span}\,\{ v_1,v_2 \}$.  The algebra structure of $C$ is preserved by this action, so $T \subseteq G_2$; in fact, $T$ is a maximal torus.

\subsection{Roots and weights}\label{subsec:lie}

For general Lie-theoretic notions, we refer to \cite{humphreys-gps}; here we explain the relevant facts for type $G_2$.  Let $G_2$ be the automorphism group of an octonion algebra $C$, as above, let $T\subset B\subset G_2$ be a maximal torus and Borel subgroup, and let $\liet\subset\lieb\subset\lieg_2$ be the corresponding Lie algebras.  Once a basis for $C$ has been chosen as in \eqref{v-basis}, we will always take $T$ to be the torus acting as in \eqref{t-action}.  Write $\alpha_1$ and $\alpha_2$ for the two \emph{simple roots}, with $\alpha_2$ the long root.  In terms of the weights $t_1,t_2$ of \eqref{t-weights}, we have
\begin{eqnarray}
\begin{array}{rcl}
\alpha_1 &=& t_1 - t_2, \\
\alpha_2 &=& -t_1 + 2\,t_2.
\end{array}
\end{eqnarray}
The \emph{positive roots} are
\begin{eqnarray*}
\alpha_1,\; \alpha_2,\; \alpha_1+\alpha_2,\; 2\,\alpha_1+\alpha_2,\; 3\,\alpha_1+\alpha_2,\; 3\,\alpha_1+2\,\alpha_2;
\end{eqnarray*}
the \emph{negative roots} are $-\alpha$, for $\alpha$ a positive root.

Let $P \subset G_2$ be the standard maximal parabolic subgroup omitting the long root, with Lie algebra $\liep\subset\lieg_2$.  Thus $\liep = \lieb\oplus\lieg_{-\alpha_1}$, where $\lieg_{-\alpha_1} \subset \lieg_2$ is the weight space for the negative root $-\alpha_1$.

The \emph{Weyl group} is $W = N(T)/T$, where $N(T)$ is the normalizer of $T$ in $G_2$.  It is isomorphic to the dihedral group with $12$ elements, and is generated by the \emph{simple reflections} $s = s_{\alpha_1}$ and $t = s_{\alpha_2}$, and is defined by the relations $s^2 = t^2 = (st)^6 = 1$.  There is an embedding $W \hookrightarrow S_7$ coming from the action of $G_2$ on $V \subset C$, given by
\begin{eqnarray*}
s &\mapsto& 2\; 1\; 5\; 4\; 3\; 7\; 6, \\
t &\mapsto& 1\; 3\; 2\; 4\; 6\; 5\; 7.
\end{eqnarray*}
(See \cite[\S A.3]{thesis}.)  We will sometimes treat elements of $W$ as permutations via this embedding.

\subsection{Flag bundles and Schubert loci}\label{subsec:schubert}

Let $C$ be an octonion algebra, and let $V = e^\perp \subset C$ as before.  A subspace $E\subseteq C$ is \emph{$G_2$-isotropic} if $E\subseteq V$ and $uv=0$ for all $u,v\in E$.  A maximal $G_2$-isotropic subspace has dimension $2$, and a $G_2$-isotropic flag $E_1 \subset E_2 \subset V$ (with $\dim E_i = i$) can be canonically extended to a complete flag $E_1\subset E_2 \subset E_3 \subset \cdots \subset E_7 = V$.  

The \emph{$G_2$ flag variety $\Fl_{G_2}$} parametrizes all $G_2$-isotropic flags in $V\subset C$.  It is a six-dimensional projective homogeneous space, isomorphic to $G_2/B$ for a Borel subgroup $B\subset G_2$.  The \emph{$G_2$ Grassmannian} $\Gr_{G_2}$ parametrizes two-dimensional $G_2$-isotropic subspaces of $V$; this is isomorphic to the five-dimensional homogeneous space $G_2/P$.

For an octonion bundle $C$ on $X$, with its rank $7$ subbundle $V$, there is an associated \emph{$G_2$-isotropic flag bundle} $\FFl_{G_2}(V) \to X$, as well as a \emph{$G_2$-isotropic Grassmann bundle} $\GGr_{G_2}(V) \to X$.  These are (\'etale-)locally trivial fiber bundles, with fibers $\Fl_{G_2}$ and $\Gr_{G_2}$, respectively.  The flag bundle $\FFl_{G_2}$ comes with a tautological flag of subbundles $\tilde{E}_\bullet$ of $V$.

Given a complete $G_2$-isotropic flag of subbundles
\begin{eqnarray*}
F_1 \subset F_2 \subset \cdots \subset F_7 = V
\end{eqnarray*}
on $X$, the \emph{Schubert loci} in $\FFl_{G_2}(V)$ are defined by
\begin{eqnarray}
& &\OOmega_w(F_\bullet) = \{ x\in \FFl_{G_2} \,|\, \dim(\tilde{E}_p(x) \cap F_q(x)) \geq r_w(q,p) \text{ for } 1\leq p,q\leq 7 \},
\end{eqnarray}
where $\tilde{E}_\bullet$ is the tautological flag on $\FFl_{G_2}$, and for $w\in W$, $r_w(q,p)$ is \mbox{$\#\{i\leq q \,|\, w(8-i) \leq p \}$}.  (Here we are using the embedding $W\hookrightarrow S_7$ discussed above.  This definition of $r_w$ differs slightly from that of \cite{g2chern}; the two are related by a factor of $w_0$.)  The codimension of $\OOmega_w$ is the \emph{length} of $w$, i.e., the least number of simple transpositions needed to write $w$ as a word in $s$ and $t$.

If $E_\bullet$ is a second $G_2$-isotropic flag on $X$, it defines a section $s_{E_\bullet}:X \to \FFl_{G_2}$ such that $s_{E_\bullet}^*\tilde{E}_\bullet = E_\bullet$.  We define degeneracy loci in $X$ as the scheme-theoretic inverse images of Schubert loci:
\begin{eqnarray*}
\Omega_w(E_\bullet,F_\bullet) = s_{E_\bullet}^{-1}\OOmega_w(F_\bullet).
\end{eqnarray*}

Proofs of all these facts, with more details, can be found in \cite{thesis} and \cite{g2chern}.  (There the term ``$\gamma$-isotropic'' is used instead of ``$G_2$-isotropic,'' in reference to a trilinear form $\gamma$.)

\section{Triality symmetry}\label{sec:triality}

Triality symmetry is described in terms of coordinates as follows.  Assume $X$ is a point, so $E$ is a two-dimensional vector space.  Choose a basis $\{v_1,v_2\}$ for $E$, and let $\{v_3,\ldots,v_8\}$ be a basis for $\End(E) \oplus E^*$ as in \eqref{v-basis}.  Suppose $\phi:E \to \End(E) \oplus \exterior^2 E^*$ is given by $\phi = \phi_1\oplus \phi_2$, with
\begin{eqnarray*}
\phi_1(v_1) &=& \left(\begin{array}{cc} a_1 & b_1 \\ c_1 & d_1\end{array}\right), \\
\phi_1(v_2) &=& \left(\begin{array}{cc} a_2 & b_2 \\ c_2 & d_2\end{array}\right), 
\end{eqnarray*}
and $\phi_2(v_1) = z\, v_2^*$, $\phi_2(v_2) = -z\, v_1^*$.  In terms of the chosen bases for $E$ and $\End(E)\oplus E^*$, $\phi$ has matrix $A_\phi^t$, where
\begin{eqnarray}\label{eqn:a-mtx}
A_\phi=\left(\begin{array}{cccccc}
b_1 & a_1 & d_1 & c_1 & z & 0 \\
b_2 & a_2 & d_2 & c_2 & 0 & -z \end{array}\right).
\end{eqnarray}

Identify $\Hom(E,\End(E)) = E^*\otimes E^*\otimes E$ with $E^*\otimes E^*\otimes E^*$ by mapping
\begin{eqnarray*}
v_i^*\otimes v_j^* \otimes v_1 &\mapsto& v_{ij2}^*, \\
v_i^*\otimes v_j^* \otimes v_2 &\mapsto& -v_{ij1}^*,
\end{eqnarray*}
where $v_{ijk}^* = v_i^*\otimes v_j^* \otimes v_k^*$ for $1\leq i,j,k \leq 2$.  (The sign appears because of the canonical isomorphism $E^*\otimes E^*\otimes E \isom E^*\otimes E^*\otimes E^* \otimes \exterior^2 E$; we are using $v_1 \wedge v_2$ to identify $E \isom E^*\otimes \exterior^2 E$ with $E^*$.)  Thus $\phi$ is triality-symmetric iff the corresponding coordinates of $v^*_{ijk}$ are invariant under permutations of the indices.  Explicitly, there is an $S_3$-action on $\Hom(E,\End(E))\oplus\exterior^2 E^*$ generated by elements $\tau$ and $\sigma$, whose action on matrices $A_\phi$ is given by
\begin{eqnarray*}
\tau\left(\begin{array}{cccccc}
b_1 & a_1 & d_1 & c_1 & z & 0 \\
b_2 & a_2 & d_2 & c_2 & 0 & -z \end{array}\right)
=
\left(\begin{array}{cccccc}
-d_2 & -c_2 & -a_1 & c_1 & z & 0 \\
b_2 & b_1 & -a_2 & d_1 & 0 & -z \end{array}\right)
\end{eqnarray*}
and
\begin{eqnarray*}
\sigma\left(\begin{array}{cccccc}
b_1 & a_1 & d_1 & c_1 & z & 0 \\
b_2 & a_2 & d_2 & c_2 & 0 & -z \end{array}\right)
=
\left(\begin{array}{cccccc}
a_2 & a_1 & c_2 & c_1 & z & 0 \\
b_2 & b_1 & d_2 & d_1 & 0 & -z \end{array}\right).
\end{eqnarray*}
This means that the triality-symmetric maps are those whose matrix is of the form
\begin{eqnarray} \label{eqn:phi-mtx}
A_\phi=\left(\begin{array}{cccccc}
a & -d & d & c & z & 0 \\
b & a & -a & d & 0 & -z \end{array}\right).
\end{eqnarray}
Here $a$ is also the coordinate of $v_{122}^*$, $b$ is the coordinate of $v_{222}^*$, $-c$ is the coordinate of $v_{111}^*$, and $-d$ is the coordinate of $v_{112}^*$.  Note that the $S_3$-invariants coincide with the $\tau$-invariants.

\begin{remark}
``Triality'' usually refers to several phenomena related to the $S_3$ symmetry of the $D_4$ Dynkin diagram first described by Cartan \cite{cartan}; see \cite{kmrt} for a thorough discussion.  The connection with our context can be explained briefly as follows.  Automorphisms of the $D_4$ Dynkin diagram correspond to outer automorphisms of the simply-connected group $\mathit{Spin}_8$; these all fix a parabolic subgroup $P$, and therefore define automorphisms of $\mathit{Spin}_8/P \isom OG(2,8)$ and the tangent space $T_{eP}\mathit{Spin}_8/P$.  The tangent space can be identified with matrices as in \eqref{eqn:a-mtx}, and under this identification, the automorphism group $S_3$ acts as descibed above.
\end{remark}

\section{Graphs}\label{sec:graphs}

For any morphism $\phi:E \to F$, let $E_\phi \subset E\oplus F$ be its graph, i.e., the subbundle whose fiber over $x$ is $E_\phi(x)= \{(v,\phi(v)) \,|\, v\in E(x)\}$.  If $\phi:E\to E^*$ is symmetric, then its graph is isotropic for the \emph{canonical skew-symmetric form} on $E\oplus E^*$, defined by $(v_1\oplus f_1, v_2\oplus f_2) = f_1(v_2)-f_2(v_1)$.  Thus one obtains a map to the Lagrangian bundle of isotropic flags in $E\oplus E^*$, and formulas for the degeneracy loci of $\phi$ are deduced from formulas for Schubert loci; see \cite{orthosymp} or \cite{fp}.

In this section, we consider morphisms $\phi:E\to End(E)\oplus E^*$.  By Proposition \ref{prop:bundle-constr}, there is a canonical octonion algebra structure on $E\oplus End(E)\oplus E^*$.  We give formulas for degeneracy loci of morphisms whose graphs are $G_2$-isotropic with respect to this structure.  In general such morphisms are not triality-symmetric (nor vice-versa).  For rank $1$ maps, however, the two notions agree. 

\begin{lemma} \label{lemma:rk1}
Suppose $X$ is a point, and $\phi:E\to\End(E)\oplus E^*$ is a triality-symmetric map, with matrix $A_\phi^t$ as in \eqref{eqn:phi-mtx}:
\begin{eqnarray*}
A_\phi = \left(\begin{array}{cccccc}
a & -d & d & c & z & 0 \\
b & a & -a & d & 0 & -z \end{array}\right).
\end{eqnarray*}
Then the graph $E_\phi$ is contained in $V\subset C$, and is $G_2$-isotropic if and only if
\begin{eqnarray} \label{eqn:rk1}
a^2 + bd = ac + d^2 = ad - bc = 0.
\end{eqnarray}
\end{lemma}

\begin{proof}
This is a straightforward verification, using the basis $\{v_i\}$ as in \eqref{v-basis}.  After a suitable change of coordinates (including a switch to opposite Schubert cells), the parametrization of the open Schubert cell given in \cite[\S D.1]{thesis} becomes
\begin{eqnarray}
\tilde\Omega^o = \left(\begin{array}{cccccccc}
1 & 0 & a & -d & d & c & z & -X \\
0 & 1 & b & a & -a & d & -Z & -Y \end{array}\right),
\end{eqnarray}
where $X = -ac-d^2$, $Y = z + ad - bc$, and $Z = -a^2 - bd$.  It is clear that the row span is always in $V\subset C$, since the fourth and fifth columns add to zero.  The condition that the row span be the graph $E_\phi$ means $X=Z=0$ and $Y=z$, which are precisely the equations \eqref{eqn:rk1}.
\end{proof}

\begin{corollary} \label{cor:rk1}
Let $\phi:E \to End(E)\oplus E^*$ be a morphism of rank at most $1$, and such that the component $\phi_2:E\to E^*$ is zero.  Then $\phi$ is triality-symmetric if and only if $E_\phi \subset C$ is $G_2$-isotropic.  (This holds scheme-theoretically, i.e., the equations locally defining these two subsets of $\Hom(E,\End(E))$ are the same.)
\end{corollary}

\begin{proof}
This is a local statement, so we may assume $X$ is a point and compute in coordinates.  In this case, it follows from Lemma \ref{lemma:rk1} by adding the equation $z=0$.  (The rank condition is forced by $\phi_2\equiv 0$.)
\end{proof}

Corollary \ref{cor:rk1} says that the formulas of Theorem \ref{thm:triality-formula} (for triality-symmetric morphisms) will agree with formulas for morphisms with $G_2$-isotropic graphs.

\begin{proof}[First proof of Theorem \ref{thm:triality-formula}]
Let $\phi:E \to End(E) \oplus E^*$ have $G_2$-isotropic graph $E_\phi$.  Suppose $E$ has a rank $1$ subbundle, so $E_\phi$ also does.  (One can always arrange for this, by passing to a $\P^1$-bundle if necessary.)  Write $E_1 \subset E_2 = E$ and $F_1\subset F_2 = E_\phi$, and extend these to complete $G_2$-isotropic flags $E_\bullet$ and $F_\bullet$.  For $w\in W$, set $\Omega_w(\phi) = \Omega_w(E_\bullet, F_\bullet)$.  Since $E_\phi \isom E$, the Chern classes are the same.  Let $-x_1,-x_2$ be Chern roots of $E$ (so $x_1,x_2$ are Chern roots of $E^*$).  Then by \cite[Theorem 2.4 and \S2.5]{g2chern}, we have
\begin{eqnarray}
[\Omega_w(\phi)] = P_w(x_1,x_2;-x_1,-x_2)
\end{eqnarray}
in $H^*X$, where $P_w(x_1,x_2;y_1,y_2)$ is the ``$G_2$ double Schubert polynomial'' defined in \cite{g2chern}.

It remains to determine the $w$ for which $D_r(\phi) = \Omega_w(\phi)$.  We have
\begin{eqnarray*}
D_r(\phi) = \{x\in X \,|\, \dim(E(x) \cap E_\phi(x)) \geq 2-r \},
\end{eqnarray*}
and it is easy to check that
\begin{eqnarray*}
D_2(\phi) &=& \Omega_{id}(\phi)  = X, \\ 
D_1(\phi) &=& \Omega_{tst}(\phi) , \\
D_0(\phi) &=& \Omega_{tstst}(\phi) .
\end{eqnarray*}
Indeed, the element $tst\in W$ corresponds to the permutation $3\;6\;1\;4\;7\;2\;5$, so the condition defining $\Omega_{tst}$ is $\dim(E_2\cap F_2) \geq r_{tst}(2,2) = 1$.  The other two identities are clear.  This also justifies our definition of expected codimension for triality-symmetric degeneracy loci: the expected codimension of $D_r(\phi)$ is the length of the corresponding element of $W$.

Specializing the polynomials $P_w$ given in \cite[\S D.2]{thesis} for these three $w$'s, we obtain the desired formulas.
\end{proof}

\begin{remark}
The twelve polynomials $P_w(x_1,x_2;-x_1,-x_2)$ become the equivariant localizations of Schubert classes in $G_2/B$ at the point $eB$, after the substitution $x_i = -t_i$.  See \cite[\S D.3]{thesis}.
\end{remark}

\section{Orbits}\label{sec:orbits}

Another approach to the computation of triality-symmetric degeneracy loci is as follows.  Inside the vector bundle $\left(Sym^3 E^* \otimes \exterior^2 E\right) \oplus \exterior^2 E^* \subset Hom(E, End(E)\oplus E^*)$, there is a locus $\mathbf{D}_r$ consisting of morphisms of rank at most $r$.  By definition, a triality-symmetric morphism $\phi$ defines a section $s_\phi$ of $\left(Sym^3 E^* \otimes \exterior^2 E\right) \oplus \exterior^2 E^*$, and $D_r(\phi) = s_\phi^{-1}\mathbf{D}_r$ is the scheme-theoretic preimage.

It suffices to solve this problem on the classifying space for the vector bundle $E$ (or on algebraic approximations thereof), so let $X = BGL_2$.\footnote{Topologically, we may assume $E$ is pulled back from the tautological bundle on $Gr(2,n)$, for $n\gg 0$, so one can take a Grassmannian for an approximation to $BGL_2$.}  Replace $E$ with the standard representation of $GL_2$, and write
\begin{eqnarray*}
U = \left(\Sym^3 E^* \otimes \exterior^2 E\right) \oplus \exterior^2 E^*.
\end{eqnarray*}
The relevant vector bundle on $BGL_2$ is $U \times^{GL_2} EGL_2$, where $EGL_2 \to BGL_2$ is the universal principal $GL_2$-bundle.  Letting $D_r \subseteq W \subset \Hom(E,\End(E)\oplus E^*)$ be the locus of maps of rank at most $r$, we have
\begin{eqnarray*}
\mathbf{D}_r = D_r \times^{GL_2} EGL_2 \subseteq U \times^{GL_2} EGL_2.
\end{eqnarray*}
Therefore $[\mathbf{D}_r] = [D_r]^{GL_2}$ in $H^*(U \times^{GL_2} EGL_2) = H_{GL_2}^*(U)$, and the problem becomes a computation in the equivariant cohomology of the vector space $U$.

Moreover, as we shall see below, $D_r$ is an orbit closure for the action of $GL_2$ on $U$.  In fact, we will use a larger group action.  As discussed in \S\ref{sec:intro}, $U$ may be identified with the tangent space
\begin{eqnarray*}
T_{[E]}G_2/P \isom \lieg_2/\liep,
\end{eqnarray*}
so $P$ acts on $U$ via the adjoint action on $\lieg_2/\liep$.  Let $P = L\cdot P_u$ be the Levi decomposition, with $P_u$ the unipotent radical and $L$ a Levi subgroup; $L$ is isomorphic to $GL_2$.  We will be interested in $P$-orbit closures in $\lieg_2/\liep$.

The $L$-action on $\lieg_2/\liep$ is identified with the natural $GL_2$-action on $U$: as an $L$-module, we have
\begin{eqnarray*}
\lieg_2/\liep \isom \left( \Sym^3 E^* \otimes \exterior^2 E \right) \oplus \exterior^2 E^*,
\end{eqnarray*}
where $E\isom \C^2$ is the standard representation of $L \isom GL_2$ (with weights $t_1=2\alpha_1+\alpha_2$ and $t_2=\alpha_1+\alpha_2$).  As a $P$-module, $\lieg_2/\liep$ does not split, but there is an exact sequence
\begin{eqnarray*}
0 \to \Sym^3 E^* \otimes \exterior^2 E \to \lieg_2/\liep \to \exterior^2 E^* \to 0.
\end{eqnarray*}
These facts follow directly from the weight decomposition of $\lieg_2/\liep$:
the $T$-weights are
\begin{eqnarray}
-\alpha_2,\; -\alpha_1-\alpha_2,\; -2\alpha_1-\alpha_2,\; -3\alpha_1-\alpha_2,\; -3\alpha_1-2\alpha_2.
\end{eqnarray}

As a first step to computing the classes of $P$-orbits in $H_T^*(\lieg/\liep)$, we give explicit descriptions of these orbits.

By the classification given in \cite{jr}, there are finitely many $P$-orbits on $\lieg/\liep$.  In fact, there are five orbits.  To describe them, let
\begin{eqnarray*}
U'=\Sym^3 E^* \otimes \exterior^2 E \subset U = \lieg_2/\liep.
\end{eqnarray*}
Let $b,a,d,c$ be coordinates on $U'$, with weights $-\alpha_2,-\alpha_1-\alpha_2,-2\alpha_1-\alpha_2,-3\alpha_1-\alpha_2$, respectively.  The five orbits are $O_c$, with $c=0,1,2,3,5$ giving the codimension; their closures are nested and described by the following proposition:

\begin{proposition} \label{prop:orbit-desc}
The $P$-orbit closures in $U=\lieg_2/\liep$ are as follows:
\begin{itemize}
\item $\overline{O_0} = U$.
\item $\overline{O_1} = U'$.
\item $\overline{O_2}$ is the discriminant locus in $U'$, defined by the vanishing of the quartic polynomial $a^2 d^2 + 4a^3c + 4bd^3 - 27b^2c^2 + 18abcd$.
\item $\overline{O_3}$ is the (affine) cone over the twisted cubic curve in $\P^3 = \P U'$, defined by the condition that the matrix
\begin{eqnarray*}
\left(\begin{array}{ccc} a & -d & c \\ b & a & d \end{array}\right)
\end{eqnarray*}
have rank $1$.
\item $\overline{O_5} = O_5 = \{0\}$.
\end{itemize}
\end{proposition}

\begin{proof}
The first claim is that $U\setminus U' = O_0$ is a single dense orbit.  This follows from the classification of \cite[Table 2]{buhe}.

It remains to verify the orbit decomposition of $U'$.  From the weights, we see that $P_u$ acts trivially on $U'$, so the effective action is by $P/P_u \isom GL_2$.  Identify $U'$ with the space of homogeneous cubic polynomials in two variables: $U = \{ -cx^3 - dx^2 y + ax y^2 + by^3\}$, with $GL_2$ acting so that the weights on $a,b,c,d$ are as specified before the proposition.  We see that there are four orbits in $U'$: the polynomials with distinct roots, those with a double root, those with a triple root, and the zero polynomial.  The given equations for the closures of these loci are well known; see e.g., \cite[IV, Ex. 12(b)]{lang} for the discriminant and \cite[\S1.1]{fp} for the cubic curve.  The proposition follows.
\end{proof}

From the description in terms of cubic polynomials, it is easy to find representatives for orbits in $U'$.  Here we give representatives as weight vectors in $\lieg/\liep$.  Let $Y_\alpha \in \lieg_2/\liep$ be a weight vector for $\alpha$.  We have
\begin{eqnarray*}
O_0 &=& P\cdot Y_{-3\alpha_1-2\alpha_2} = U \setminus U'; \\
O_1 &=& P\cdot (Y_{-3\alpha_1-\alpha_2}+Y_{-\alpha_2}) \isom P/P_u \isom GL_2 ; \\
O_2 &=& P\cdot Y_{-\alpha_1-\alpha_2}; \\
O_3 &=& P\cdot Y_{-\alpha_2}; \\
O_5 &=& \{0\}.
\end{eqnarray*}

Using Proposition \ref{prop:orbit-desc}, it is a simple matter to compute the equivariant classes.
\begin{theorem} \label{thm:orbit-formula}
In $H_T^*(U) = \Z[\alpha_1,\alpha_2] = \Z[t_1,t_2]$, we have
\begin{eqnarray*}
{[\overline{O_0}]} &=& 1  \\
{[\overline{O_1}]} &=& -3\alpha_1-2\alpha_2 \\ &=& -t_1-t_2  \\
{[\overline{O_2}]} &=& 2(-3\alpha_1-2\alpha_2)^2 \\ &=& 2(t_1+t_2)^2 ; \\
{[\overline{O_3}]} &=& -3(\alpha_1+\alpha_2)(2\alpha_1+\alpha_2)(3\alpha_1+2\alpha_2) \\ &=& -3t_1 t_2 (t_1+t_2) ; \\
{[\overline{O_5}]} &=& -\alpha_2(\alpha_1+\alpha_2)(2\alpha_1+\alpha_2)(3\alpha_1+\alpha_2)(3\alpha_1+2\alpha_2) \\ &=& t_1 t_2 (t_1+t_2)(2t_1-t_2)(t_1-2t_2).
\end{eqnarray*}
\end{theorem}

\begin{proof}
The normal space to $U'=\overline{O_1}\subset U$ has weight $-3\alpha_1-2\alpha_2$, so the formula for $[\overline{O_1}]$ is clear.  Since the restriction $H_T^*(U) \to H_T^*(U')$ is an isomorphism, the Gysin pushforward $H_T^*(U')\to H_T^*(U)$ is multiplication by $[U']$.  Therefore it suffices to compute the remaining classes in $H_T^*(U')$.  The locus $\overline{O_2}$ is a hypersurface in $U'$ defined by an equation of weight $-6\alpha_1-4\alpha_2$, so its class in $H_T^*(U)$ is $(-6\alpha_1-4\alpha_2)\cdot [U']$.  The class of $[\overline{O_3}]$ in $H_T^*(U')$ is found by the classical Giambelli (or Salmon--Roberts) formula (see e.g. \cite[\S1.1]{fp}).  Finally, the class of the origin is the product of all the $T$-weights on $U$.
\end{proof}

\begin{remark} \label{rmk:restriction}
These classes cannot be computed by the ``restriction equation'' method of Feh\'er and Rim\'anyi \cite{fr}, because the stabilizer of $O_1 = P/P_u$ is unipotent.  This means the restriction map $H_P^*(U) \to H_P^*(O_1) \isom H_{P_u}^*(pt) = H^*(pt)$ is zero in positive degees, and all the restriction equations are of the form $0=0$.  (The problem persists for the other orbits.)
\end{remark}

\begin{lemma} \label{lemma:orbit-rank}
The orbit
\begin{eqnarray*}
O_3 \subset \lieg/\liep \subset \Hom(E,\End(E)\oplus E^*)
\end{eqnarray*}
consists of the triality-symmetric morphisms of rank $1$.
\end{lemma}

\begin{proof}
First note that any rank $1$ map $\phi$ must correspond to an element $\phi_1\oplus \phi_2 \in U = U'\oplus \exterior^2 E^*$ with $\phi_2 = 0$, i.e., $\phi$ lies in $U'$.  (If $\phi_2 \neq 0$, then $\phi$ surjects onto $E^*$.)

Now the action of $P$ on $U'$ is the same as that of its Levi subgroup $GL_2$.  The inclusion $P\hookrightarrow P_{\hat{2}}\subset GL_8$ induces an inclusion of Levi subgroups $GL_2 \hookrightarrow GL_2 \times GL_6$, and the latter acts by conjugation on $\Hom(E,\End(E)\oplus E^*)$, so it preserves ranks of morphisms.  Therefore it will suffice to check that a representative for $O_2$ has rank $2$, and a representative from $O_3$ has rank $1$.  For these, we use the coordinate description given in \S\ref{sec:triality}.  Under the identification of $U'$ with the space of cubic polynomials, the monomial $x y^2$ corresponds to the basis vector $v_{122}^*$.  The orbit is $O_2$ (since $x y^2$ has two distinct zeroes), and the corresponding matrix $A_\phi$ has $b=c=d=0$ and $a\neq 0$; it is easy to see this means $\phi$ has rank $2$.  Similarly, $x^3$ corresponds to $v_{111}^*$, and the corresponding $A_\phi$ has $a=b=d=0$ and $c\neq 0$, so $\phi$ has rank $1$.
\end{proof}

The formulas of Theorem \ref{thm:triality-formula} now follow from those of Theorem \ref{thm:orbit-formula}.

\begin{proof}[Second proof of Theorem \ref{thm:triality-formula}]
Let $f:X \to BGL_2$ be the map defined (up to homotopy) by the given vector bundle $E$ on $X$.  The corresponding map
\begin{eqnarray*}
f^*: H^*BGL_2 = H_{GL_2}^*(pt) = \Z[c_1,c_2] \to H^*X
\end{eqnarray*}
is given by $c_i \mapsto c_i(E) = (-1)^i c_i(E^*)$.  Equivalently, using the inclusion $H_{GL_2}^*(pt) \subset H_T^*(pt) = \Z[t_1,t_2]$ and Chern roots $x_1,x_2$ for $E^*$, the map is given by $t_i \mapsto - x_i$.

Now using Lemma \ref{lemma:orbit-rank}, we have $f^*[\bar{O}_3] = [D_1(\phi)]$ when $D_1(\phi)$ has expected codimension, and Theorem \ref{thm:triality-formula} follows from Theorem \ref{thm:orbit-formula}.
\end{proof}

\appendix

\section*{Appendix: Octonion bundles}\label{sec:octonions}

\renewcommand{\thetheorem}{A.\arabic{theorem}}

In this appendix, we establish a $G_2$ analogue of the well-known fact that for any vector bundle $E$, the direct sum $E\oplus E^*$ carries canonical symplectic (type $C$) and symmetric (type $D$) forms; see e.g. \cite[p. 71]{fp}.  This construction does not seem to appear explicitly in the literature, although it is closely related to the Cayley--Dickson doubling construction (\cite{petersson}); see also \cite{mukai}.

We fix some notation.  For any vector bundle $E$, let 
\[
  Tr:End(E)=E^*\otimes E \to \O_X
\]
be the canonical contraction map, and let
\[
  End^0(E) = \ker(Tr) \subset End(E)
\]
be the subbundle of trace-zero endomorphisms.  Let $e:\O_X \to End(E)$ be the identity section.  Thus the composition $Tr\circ e:\O_X \to \O_X$ is multiplication by $\rk(E)$.  Also, the \emph{conjugation} map $End(E) \to End(E)$ is given by $e\circ Tr - id$.  (Here $id$ is the identity morphism, as opposed to the identity section $e$.)  Conjugation is an involution; locally, it is $\xi \mapsto \bar\xi := Tr(\xi)e - \xi$.

The norm on an octonion bundle $C$ corresponds to a nondegenerate symmetric bilinear form $\<\;,\;\>$.  Let $V \subset C$ be the orthogonal complement to the identity subbundle defined by $e$.  A subbundle $E \subset C$ is \emph{$G_2$-isotropic} if it is contained in $V$ and the multiplication map $E\otimes E \to C$ is the zero map.

\begin{proposition} \label{prop:bundle-constr}
Let $E$ be a rank $2$ vector bundle on a variety $X$.  Then $C = E \oplus End(E) \oplus E^*$ has a canonical octonion bundle structure, with identity section $e:\O_X \to End(E) \subset C$.  The subbundle $E=E\oplus 0\oplus 0 \subset C$ is $G_2$-isotropic.
\end{proposition}

\begin{proof}
We need to define the norm $N:C \to \O_X$ and multiplication $m:C\otimes C \to C$, for $C=E\oplus End(E) \oplus E^*$, and check that they are compatible.

The norm on $C$ corresponds to the bilinear form $\<\;,\;\>$ defined by
\begin{eqnarray*}
\< x\oplus \xi\oplus f,\, y\oplus \eta\oplus g \> = Tr(\xi)Tr(\eta)-Tr(\xi\eta)-f(y)-g(x).
\end{eqnarray*}
(This can also be expressed in terms of natural contraction maps.)  It is clear that $\<\;,\;\>$ is nondegenerate.  Thus
\begin{eqnarray*}
N(x\oplus \xi\oplus f) = \det(\xi) - f(x)
\end{eqnarray*}
is a nondegenerate quadratic norm on $C$.  We also see that $V = E \oplus End^0(E) \oplus E^*$.

The multiplication is given by
\begin{eqnarray*} 
(x\oplus \xi\oplus f)\cdot(y\oplus \eta\oplus g) &=& (\eta x + \bar\xi y) \oplus ( \bar{g\otimes x} + \xi \eta + f\otimes y) \oplus (g\xi + f\bar\eta).
\end{eqnarray*}
Noting that $\bar{e}=e$, it is easy to see that $e$ (the identity for $End(E)$) acts as a multiplicative identity for $C$.  Moreover, the multiplication restricts to zero on $E\oplus 0 \oplus 0 \subset C$.

To verify the multiplicativity of the norm, we compute:
\begin{eqnarray*}
N( (x\oplus \xi\oplus f)\cdot(y\oplus \eta\oplus g) ) &=& \det( \bar{g\otimes x} + \xi \eta + f\otimes y ) - ( g\xi + f\bar\eta )(\eta x + \bar\xi y) \\
&=& \det(\xi\eta) + \< \bar{g\otimes x}, \xi\eta \> + \< \bar{g\otimes x}, f\otimes y \>  \\
& & + \< \xi\eta, f\otimes y \> - ( g\xi\eta x + g\xi\bar\xi y + f\bar\eta\eta x + f\bar\eta\bar\xi y) \\
&=& \det(\xi)\det(\eta) + \< g(x)e, \xi\eta \> - \< g\otimes x, \xi\eta \> \\
& & + \< g(x)e, f\otimes y \> - \< g\otimes x, f\otimes y \> +  \< \xi\eta, f\otimes y \> \\
& & - g\xi\eta x - \det(\xi) g(y) - f(x)\det(\eta) - f\bar{\xi\eta} y \\
&=& \det(\xi)\det(\eta) + g(x) Tr(\xi\eta) - g(x) Tr(\xi\eta) + g\xi\eta x \\
& & + g(x) f(y) - g(x) f(y) + g(y) f(x) + Tr(\xi\eta) f(y)  \\
& &  - f\xi\eta y - g\xi\eta x - \det(\xi) g(y) - f(x)\det(\eta) - f\bar{\xi\eta} y \\
&=& \det(\xi)\det(\eta) - \det(\xi) g(y) - f(x)\det(\eta) + f(x) g(y) \\
& & + Tr(\xi\eta) f(y) - f\xi\eta y - f(Tr(\xi\eta) e - \xi\eta)y \\
&=& \det(\xi)\det(\eta) - \det(\xi) g(y) - f(x)\det(\eta) + f(x) g(y) \\
&=& N(x\oplus \xi \oplus f)\, N(y\oplus \eta \oplus g).
\end{eqnarray*}
Thus we have defined an octonion algebra structure on $C$.  

Since the multiplication is zero on $E=E\oplus 0\oplus 0$, it follows that $E\subset V$ is $G_2$-isotropic.
\end{proof}

\raggedbottom

\end{document}